\documentclass[]{article}
\usepackage{amsmath, amsthm, amsfonts, amssymb, enumitem}
\usepackage[english]{babel}
\newtheorem{thm}{Theorem}[section]
\newtheorem{cor}[thm]{Corollary}
\newtheorem{lem}[thm]{Lemma}
\newtheorem{obs}[thm]{Observation}
\theoremstyle{remark}
\newtheorem{rem}[thm]{Remark}

\newcommand{\dom}{\mathrm{dom}}

\newcommand{\NN}{\mathbb{N}}
\newcommand{\ZZ}{\mathbb{Z}}

\newcommand{\RR}{\mathbb{R}}
\newcommand{\Eins}{\mathbf{1}}
\newcommand{\cH}{\mathcal{H}}
\newcommand{\cK}{\mathcal{K}}

\newcommand{\ran}{\mathrm{ran}}
\newcommand{\set}[2]{\bigl\{#1\bigm|#2\bigr\}}
\begin{document}
	\begin{center}
		{\LARGE \bf Uncertainty principles and lower bounds\\[0.5ex] for Schr\"odinger operators}
		
		\vspace{0.5cm}
		{\large A. B\"ottcher\footnote{Fakult\"{a}t f\"{u}r Mathematik,  Technische
				Universit\"{a}t Chemnitz, 09107 Chemnitz, Germany, 
				{\tt aboettch@mathematik.tu-chemnitz.de}},
			P. Stollmann\footnote{Fakult\"{a}t f\"{u}r Mathematik,  Technische
				Universit\"{a}t Chemnitz, 09107 Chemnitz, Germany, 
				{\tt P.Stollmann@mathematik.tu-chemnitz.de}},
			and M. Tautenhahn\footnote{Mathematisches Institut,
				Universit\"at Leipzig,  Augustusplatz 10, 04109 Leipzig, Germany,
				{\tt martin.tautenhahn@math.uni-leipzig.de}}}
		
	\end{center}
	
	\vspace{0.2cm}
	\begin{abstract}
		\noindent
		We prove that two different abstract quantitative uncertainty principles are equivalent to the strict positivity of the associated abstract Schr\"o\-din\-ger operators. We also discuss the case of continuum Schr\"odinger operators, in which case our method provides an explicitly computable lower bound as well as a control theoretic application.
	\end{abstract}
	
	\section{Introduction}
	
	This article is concerned with uncertainty principles at low energy for Schr\"o\-ding\-er operators. 
	A typical manifestation of such an uncertainty principle might be
	\begin{equation} \label{eq:intro1}
		P_E V P_E \geq \kappa P_E ,
	\end{equation}
	where, in a Hilbert space $\mathcal{H}$, $P_E=\boldsymbol{1}_{[0,E]}(H)$ is the spectral projector associated 
	to some selfadjoint operator $H\ge 0$ onto energies below $E \geq \min \sigma (H)$, 
	$V$ is a non-negative selfadjoint operator, and $\kappa = \kappa (E)$ is some positive constant. 
	The most prominent example is the case where $H=-\Delta$ is the negative Laplacian in $L^2 (\RR^d)$ and $V = \Eins_D$ with a so-called thick set $D \subset \RR^d$. In that case, the Logvinenko-Sereda theorem implies \eqref{eq:intro1} for all $E > 0$ with an explicit constant $\kappa$ depending on $E$ and the geometric properties of the set $D$; see \cite{LogvinenkoS-74,Kovrijkine-00,Kovrijkine-01}. It is called an uncertainty principle since it says that 
	$$
	\| f\Eins_D\|^2\ge \| f\|^2\mbox{  for all }f\in\ran(P_E) , 
	$$
	so that 
	states of low energy are spread out in configuration space in the sense that they carry a reasonable amount of mass on the set $D$. The proof of the Logvinenko-Sereda theorem strongly relies on harmonic analysis. The presence of a potential and of more general domains $\Omega \subset \RR^d$, i.e., considering $H = -\Delta + W$ in $L^2 (\Omega)$, places the problem beyond the reach of purely harmonic analysis methods. This setting has been investigated, e.g., in \cite{Klein-13,RojasMolinaV-13,NakicTTV-18,LebeauM-19,DickeSV-24} under different assumptions on the set $D$ and the potential $W$.
	
	\medskip
	From the viewpoint of applications, uncertainty relations of the type \eqref{eq:intro1} have become an essential tool
	in spectral theory of random Schr\"odinger operators, control theory, or harmonic analysis, to mention a few. In the theory of random Schr\"odinger operators such uncertainty relations are typically used to prove Wegner estimates and in turn Anderson localization; see, for example, \cite{BoutetdeMonvelLS-11,Klein-13,RojasMolinaV-13,NakicTTV-18,StollmannS-21}. In control theory, 
	uncertainty relations serve as an input to the Lebeau-Robbiano strategy \cite{LebeauR-95}, originally formulated for the Laplace-Beltrami operator on compact Riemannian manifolds~$M$. The general philosophy here is that an uncertainty relation implies an observability estimate, and by Douglas' lemma \cite{Douglas-66}, this is equivalent to null-controllability of a certain linear control problem. Subsequently, this strategy has been studied in various mathematical situations, including abstract Hilbert or Banach spaces; see, e.g., \cite{LebeauZ-98,JerisonL-99,Miller-10,TenenbaumT-11,WangZ-17,BeauchardP-18,NakicTTV-20,GallaunST-20}. While the original Lebeau-Robbiano strategy requires \eqref{eq:intro1} to hold for all $E > 0$ with a constant $\kappa = \kappa (E)$ not vanishing too fast as $E$ tends to infinity, there exist also weak versions where the validity of \eqref{eq:intro1} for some (eventually small) fixed $E$ is sufficient; cf. the recent papers \cite{EgidiGST-24,MuenchSST}.
	
	\medskip
	Let us now turn to our contribution. Our key result is Theorem \ref{Theo3.2}, which gives a computable lower bound for abstract Schr\"odinger
	operators. This theorem plays a crucial role in our proof of the equivalence of the uncertainty principle known 
	from~\cite{CombesHK-03,CombesHK-07,BoutetdeMonvelLS-11}  and the strict positivity of Schr\"odinger operators. Moreover, this theorem
	allows us to pass from potentials of the form $V=\Eins_D$ with relatively dense sets $D$ to 
	non-negative potentials $V$ subject to what we call $(R,\varrho)$-thickness. 
	
\medskip
	It is evident that the technique employed here could equally well be applied in the discrete case of graph Laplacians;
	we refer to \cite{LenzSS-19} for earlier results on graphs and to the upcoming \cite{LenzS}, which features a metric measure 
	space set-up that includes graphs and continuum configuration spaces.  
	
	\medskip
	In contrast to the equivalent quantitative uncertainty principles mentioned above, 
	a qualitative uncertainty principle for $H$ and $V$ can be thought of as a unique-continuation statement: 
	for suitable $E$, it requires that if $f\in\ran(P_E)$
	and $Vf=0$, then necessarily $f=0$. It is a rather simple observation  that such a unique-continuation 
	is a consequence of the quantitative uncertainty principles mentioned earlier. For $H=-\Delta$, the fact that functions 
	in the range of the spectral projector are analytic shows that in this case 
	the qualitative uncertainty principle holds for
	arbitrary $E$ and every $V\ge 0$ that is not $0$ almost everywhere, whereas a quantitative uncertainty principle requires $V$ to be $(R,\varrho)$-thick for some $R,\varrho>0$.
	\medskip
	
	In the final section we sketch an application of our method to control theory: it provides $\alpha$-controllability with arbitrary $\alpha>0$ for a new and fairly general class of control problems.

	\section{Main results}
	
	We start in a general setting: $(\mathcal{H}, (\cdot|\cdot))$ denotes a Hilbert space, $H, V$ are selfadjoint
	operators in $\mathcal{H}$, $V$ is supposed to be bounded, and we require that  $H,V \ge0$.
	 As usual, if $A$ is selfadjoint and $\varepsilon \in \RR$,
	then
	$A \ge \varepsilon$ means that $(Af|f) \ge \varepsilon (f|f)$ for
	all $f \in {\rm dom}(A)$, which in turn is equivalent to the requirement that the spectrum $\sigma(A)$ of $A$ is
	contained in $[\varepsilon,\infty)$. 
	We call $A$ {\bf strictly positive} and write $A \gg 0$ if $A \ge \varepsilon$
	for some $\varepsilon > 0$. Most of what follows is only interesting in the case
	where $0 \in \sigma(H)$, that is, where $H$ is not strictly positive.
	
	\medskip
	For $E > 0$, 
	we denote by $P_E := \boldsymbol{1}_{[0,E]}(H)$ the spectral projector of $H$ associated with $[0,E]$.
	\medskip
	Following \cite{BoutetdeMonvelLS-11}, we say that a {\bf quantitative uncertainty principle for spectral projectors} holds for the pair $H,V$ 
	at $E$ if there is a $\kappa=\kappa(H,V,E) >0$ such that
	\begin{equation}\label{sUP}\tag{$\sigma$UP}
		P_E V P_E \ge \kappa P_E,
	\end{equation} 
	where, of course, $P_E V P_E \ge \kappa P_E$ means that $P_E V P_E - \kappa P_E \ge 0$. 
	
	\medskip
	We say that a  {\bf quantitative uncertainty principle at low energy} holds for the pair $H,V$ 
	at $E$ if
	there exists a $\mu=\mu(H,V,E)>0$ such that
	\begin{equation}\label{EUP}\tag{$\mathcal{E}$UP}
		f \in {\rm dom}(H) \;\: \land \;\: (Hf|f)\le E \lVert f \rVert^2  \;\: \Longrightarrow \,\: (Vf|f) \ge \mu \lVert f\rVert^2.
	\end{equation}
	Clearly, (\ref{EUP}) implies (\ref{sUP}) with $\kappa=\mu$: since $(HP_Ef|P_Ef) \le E \lVert P_Ef\rVert^2$,
	we infer from (\ref{EUP}) that
	$(P_EVP_Ef|P_Ef) =(VP_Ef|P_Ef) \ge \mu \lVert P_Ef\rVert^2$.
	
	
	\medskip
	Here is our abstract main characterization of quantitative uncertainty principles.

	\begin{thm} \label{Theo2.1}
    The following are equivalent:
        \begin{enumerate}[label=(\roman*),font=\upshape,noitemsep,topsep=0pt]
			\item $H+V \gg 0$;
			\item $H$ and $V$ obey  {\rm (\ref{sUP})} at some $E=E_1>0$ with suitable $\kappa$;
			\item $H$ and $V$ obey  {\rm (\ref{EUP})} at some $E=E_2>0$ with suitable $\mu$.
		\end{enumerate}
	\end{thm}
	
	We now turn to the special case where $H=-\Delta: H^2(\mathbb{R}^d) \to L^2(\mathbb{R}^d)$
	is the negative Laplacian
	and $V$ is the operator of multiplication by a non-negative function in $L^\infty(\mathbb{R}^d)$.
	We denote this function by $V$, too. Given two numbers $R,\varrho >0$,
	the function $V$ is said to be $\boldsymbol{(R,\varrho)}${\bf-thick} if
	\[\frac{1}{R^d}\int_{C_R(y)} V(x) \mathrm{d}x \ge \varrho\]
	for all $y \in \mathbb{R^d}$, where $C_R(y)=y+(0,R)^d$, so that the left hand side is the mean of $V$ over $C_R(y)$.
	Using appropriate trial functions one can show that
	\begin{equation} \label{Dia}\tag{$\Diamond$}
		\inf \sigma(-\Delta+V)\le\liminf_{R \to \infty}
		\left(\inf_{y \in \mathbb{R}^d}\frac{1}{R^d}\int_{C_R(y)} V(x) \mathrm{d}x\right),
	\end{equation} 
	an estimate very reminiscent of results from~\cite{GGS92}. 
	We will give a proof of this inequality in Section~\ref{Sec4}.
	Herewith our second main result.
	
	\begin{thm} \label{Theo2.2}
		The following are equivalent:
    \begin{enumerate}[label=(\roman*),font=\upshape,noitemsep,topsep=0pt]
		\item $-\Delta+V \gg 0$;
		\item $V$ is $(R,\varrho)$-thick for some $R,\varrho$;
		\item the lower limit on the right of {\rm (\ref{Dia})} is positive.
    \end{enumerate}
	\end{thm}
	
	The only non-trivial implication of this equivalence is (ii)$\Longrightarrow$(i); the attentive reader will note that for $V=\Eins_D$ this could be readily deduced from the above mentioned Logvinenko-Sereda Theorem combined with Theorem \ref{Theo2.1}. However, we prefer to give a computable lower bound on $\inf \sigma(-\Delta+V)$ in Theorem~\ref{Theo4.1} and Corollary~\ref{Cor4.2}. These results are again based on Theorem \ref{Theo3.2}. We want to stress the fact that the previous result gives uncertainty principles for large classes of operators, namely for $H\ge \eta(-\Delta)$, $\eta>0$ and $(R,\varrho)$-thick $V$. In Subsection~5.2 we discuss a concrete class of examples that can be treated in this way.  
		\medskip
	
	We say that a {\bf qualitative uncertainty principle} holds for the pair $H,V$ at $E$ if
	\begin{equation}\label{QUP}\tag{QUP}
		f \in \ran(P_E) \;\: \land \;\: Vf=0 \;\: \Longrightarrow \,\: f=0.
	\end{equation} 
	It is easily seen
	that ($\sigma$UP) and thus each of the three equivalent conditions
	in Theorem~\ref{Theo2.1} imply (\ref{QUP}). Indeed, if $f \in P_E(\mathcal{H})$ and $Vf=0$, then \eqref{sUP} gives
	\[0=(Vf|f)=(VP_Ef|P_Ef)=(P_EVP_Ef|P_Ef)\ge \kappa (P_Ef|P_Ef)=\kappa (f|f),\]
	implying that $f=0$. The converse implication is totally wrong:
	
	\begin{obs} \label{Theo2.3} Let $H=-\Delta$, $V\in L^\infty(\RR^d)\setminus\{0\}$, $V\ge 0$. 
		Then {\rm (\ref{QUP})} holds for $H,V$ and every $E\ge 0$.
	\end{obs}
	
	The following two sections are devoted to the proofs of the two theorems stated above.
	In Section 5 we deal with Observation 2.3, show that in the Schrödinger operator set-up, we get quantitative uncertainty principles for $-\Delta$ replaced by rather general classes of operators and give some concrete examples, both of strictly positive
	operators $-\Delta+V$ and of such operators that are not strictly positive. The last sections contains a short discussion of applications of the above results to control theory.

	\section{The abstract setting}
	
	In this section {\em we prove Theorem \ref{Theo2.1}}.
	
	\medskip
	The equivalence (i) $\Longleftrightarrow$ (iii) is readily seen. Indeed, if (i) holds with $H+V\ge E_0$, then  (iii) holds with $E\in (0,E_0)$ and $\mu=E_0-E$. 
	Conversely, suppose (iii) is true. Then 
	$H+V \ge \min(E,\mu)$, which is (i).
	
	\medskip
	The obvious implication (iii) $\Longrightarrow$ (ii)  was already mentioned in the previous section.
	We are therefore left with the implication (ii) $\Longrightarrow$ (i), whose proof will
	occupy the rest of this section. We start with a simple result for $2 \times 2$ matrices.
	
	\begin{lem}\label{Lem3.1}
		Let 
		\[
		 A=\begin{pmatrix} x & z\\ \overline{z} & y \end{pmatrix} 
		 \quad\text{and}\quad 
		 B=\begin{pmatrix} \beta & w\\ \overline{w} & \delta \end{pmatrix}
		\]
		be positive matrices, $A,B \ge 0$, and suppose $A+B$ is not the zero matrix. Then $A+B$ is positive and
		\begin{equation}\label{eqAB}
			A+B \ge \frac{(x+\beta)(y+\delta)-\lvert z+w \rvert^2}{\lVert A+B \rVert}.
		\end{equation}
	\end{lem}

	\begin{proof} Since $((A+B)\xi |\xi)=(A\xi |\xi)+(B\xi |\xi)\ge 0$ for all $\xi$, we see that $A+B \ge 0$.
	Let $\lambda_1 \ge \lambda_2 \ge 0$ be the eigenvalues of $A+B$. As $\lambda_1$ and $\lambda_2$ cannot be zero 
	simultaneously,
	we have
	\[A+B \ge \lambda_2 = \frac{\lambda_1\lambda_2}{\lambda_1}=\frac{\det(A+B)}{\lVert A+B\rVert}, \]
	which is just (\ref{eqAB}).
	\end{proof}
	
	\medskip
	Recall that a closed subspace $\cK$ of $\cH$ is called \emph{reducing for $H$} 
	if $\cK\subset\dom(H)$ and $P_{\cK} H \subset H P_{\cK}$, where $P_{\cK}$ denotes 
	the orthogonal projection onto $\cK$. This is equivalent to the 
	split $\dom(H) = (\cK \cap \dom (H)) + (\cK^{\perp} \cap \dom (H))$ 
	together with the requirement that both $\cK$ and $\cK^\perp$ are invariant under $H$, i.e., for
	$\eta \in \cK \cap \dom(H)$ we have $H \eta \in \cK$ and similarly on $\cK^\perp \cap \dom(H)$.
	In particular, this implies  that $P_{\cK} \eta \in \dom (H)$ if $\eta \in \dom (H)$.
	
	\begin{thm}\label{Theo3.2}
		Let $\cK \subset \cH$ be a closed reducing subspace for $H$ such that, for some $c,v >0$,
				\[(H \xi | \xi)\ge c\lVert \xi \rVert^2\;\: \mbox{for all}\;\: \xi\in\cK^\bot \cap \dom(H),\quad
		(V\zeta |\zeta)\ge v \lVert \zeta \rVert^2 \;\: \mbox{for all}\;\:\zeta \in \cK.\] 
		Then
		$$
		H+V\geq\frac{cv}{c+2\lVert V \rVert }.
		$$
	\end{thm}
	
	\begin{proof} Let $\eta \in\dom(H)$ and decompose $\eta = \zeta + \xi$ with
	$$
	\zeta = P_\cK\eta
	, \quad 
	\xi = (I-P_\cK)\eta. 
	$$
	Since $\cK$ was assumed to be a reducing subspace, it is guaranteed that $\zeta , \xi$ belong to $\dom (H)$. 
	Assume first that both $\zeta$ and $\xi$ are not zero. We have
	\[(H\eta|\eta)=(H\zeta|\zeta)+(\zeta| H\xi)+(H\xi|\zeta)+(H\xi|\xi)=
	(H\zeta|\zeta)+(H\xi|\xi) \,\]
	and letting $\gamma:=(H\xi|\xi)/\lVert \xi\rVert^2$, this may be written as
	\[(H\eta|\eta)=(H\zeta|\zeta)+\begin{pmatrix}\lVert \zeta\rVert \!\! &\!\! \lVert \xi\rVert\end{pmatrix}
	\begin{pmatrix}0 & 0\\ 0 & \gamma\end{pmatrix} \begin{pmatrix}\lVert \zeta\rVert \\ \lVert \xi\rVert\end{pmatrix}
	=(H\zeta|\zeta)+\begin{pmatrix}\lVert \zeta\rVert \!\! &\!\! \lVert \xi\rVert\end{pmatrix}
	A \begin{pmatrix}\lVert\zeta\rVert \\ \lVert \xi\rVert\end{pmatrix}\]
	with $A=\begin{pmatrix}0 & 0\\ 0 & \gamma\end{pmatrix} \ge 0$. Put
	$$
	\beta := \left( V\frac{\zeta}{\lVert \zeta\rVert} \bigg \vert \frac{\zeta}{\lVert \zeta\rVert}\right), 
	\quad
	w:=\left( V\frac{\zeta}{\lVert \zeta\rVert} \bigg \vert \frac{\xi}{\lVert \xi\rVert}\right),
	\quad
	\delta := \left( V\frac{\xi}{\lVert \xi\rVert} \bigg \vert \frac{\xi}{\lVert \xi\rVert}\right).
	$$
	We so may write
	\[(V\eta|\eta)=\begin{pmatrix}\lVert \zeta\rVert\!\! &\!\! \lVert\xi\rVert\end{pmatrix}
	\begin{pmatrix} \beta & w\\ \overline{w} & \delta\end{pmatrix} \begin{pmatrix}\lVert\zeta\rVert \\ \lVert\xi\rVert\end{pmatrix}
	=\begin{pmatrix}\lVert\zeta\rVert\!\! &\!\! \lVert\xi\rVert\end{pmatrix}
	B \begin{pmatrix}\lVert\zeta\rVert \\ \lVert\xi\rVert\end{pmatrix}\]
	with $B=\begin{pmatrix} \beta & w\\ \overline{w} & \delta\end{pmatrix}$.
	Clearly,  $\beta,\delta \ge 0$, and Cauchy-Schwarz yields $\lvert w \rvert^2 \leq \beta \delta$.
	Thus, $B \ge 0$. 
	Lemma \ref{Lem3.1} now
	gives
	\[A+B \ge \frac{\beta(\gamma+\delta)-\lvert w\rvert^2}{\lVert A+B\rVert} \ge \frac{\beta\gamma}{\lVert A\rVert+\lVert B\rVert}\]
	Obviously, $\lVert A\rVert =\gamma$. Estimating the norm of $B$ from above by the Frobenius norm and using again
	Cauchy-Schwarz, we get
	\[\lVert B\rVert^2 \le \beta^2 +2\lvert w\rvert^2+\delta^2 \le 4 \lVert V\rVert^2.\]
	Thus,
	\[\frac{\beta\gamma}{\lVert A\rVert +\lVert B\rVert}\ge \frac{\beta\gamma}{\gamma +2\lVert V\rVert}\ge \frac{cv}{c+2\lVert V\rVert},\]
	the last inequality resulting from the assumptions $(V\zeta|\zeta)\geq v \lVert \zeta\rVert^2$
	and $(H\xi|\xi)\geq c\lVert \xi\rVert^2$, which imply that $\beta \ge v$ and $\gamma \ge c$.
	In summary, in the case where both $\zeta$ and $\xi$ are not zero,
	\begin{equation*}
		\left( (H+V) \eta | \eta \right)
		\ge 
		(H \zeta | \zeta) +
		\frac{cv}{c+2\lVert V\rVert} \lVert \eta \rVert^2
		\ge \frac{cv}{c+2\lVert V\rVert} \lVert \eta \rVert^2.
	\end{equation*}
	A direct inspection shows that this inequality (ignoring the middle term) holds as well if 
	either $\zeta = 0$ or $\xi = 0$. Indeed, if $\zeta \not = 0$ and $\xi = 0$, we use $\beta \geq v$ to obtain
	\[
	\frac{\left( (H+V) \eta | \eta \right)}{\lVert \zeta \rVert^2} 
	= (H \zeta | \zeta) + \beta \lVert \zeta \rVert^2
	\ge \beta \lVert \zeta \rVert^2 \ge v \lVert \zeta \rVert^2
	\ge 
	\frac{cv}{c+2\lVert V \rVert} \lVert \zeta \rVert^2 . 
	\]
	In the case where $\zeta = 0$ and $\xi \not = 0$, we take into account that $\lVert B\rVert \ge \beta \ge v$
	and $\gamma \ge c$ to get
	\[
	\frac{\left( (H+V) \eta | \eta \right)}{\lVert \xi \rVert^2} = \gamma + \delta
	= 
	\frac{\gamma c + \delta c + \gamma \lVert B \rVert + \delta \lVert B \rVert}{c + \lVert B \rVert}
	\ge
	\frac{cv}{c + \lVert B \rVert} \ge \frac{cv}{c + 2\lVert V \rVert}.
	\]
	This completes the proof.
	\end{proof}
	The implication (ii) $\Longrightarrow$ (i) of Theorem~\ref{Theo2.1} follows from Theorem~\ref{Theo3.2} applied to $\cK=P_E(\cH)$.
	In that case $\cK$ is a closed reducing subspace for $H$ and we have $(H\xi|\xi) > E\lVert \xi\rVert^2$ for
	$\xi$ in $\dom(H) \cap \cK^\perp$.
	The assumption $P_EVP_E \ge \kappa P_E$ gives $(V\zeta|\zeta) \ge \kappa \lVert\zeta\rVert^2$ for 
	$\zeta=P_E g$ in $\cK$. From Theorem~\ref{Theo3.2} we therefore obtain that
	\[H+V \ge \frac{\kappa E}{E+2\lVert V\rVert}.\]

	\section{The case $\boldsymbol{H=-\Delta}$} 	\label{Sec4}

	This section is devoted to the {\em proof of Theorem \ref{Theo2.2}} and inequality \eqref{Dia}. We start with the latter.
	
	\begin{proof}[Proof of \eqref{Dia}] Let $\gamma$ be strictly larger than the $\liminf$.
	Then there are  $R_n\to\infty$ and $y_n\in\RR^d$ such that
	$$
	\frac{1}{R_n^d}\int_{C_{R_n}(y_n)}V(x)\mathrm{d}x<\gamma .
	$$
	Pick $z_n=y_n+(1/2,\ldots,1/2)$, so that the cubes $C_{R_n-1}(z_n)\subset C_{R_n}(y_n)$ 
	have the same center and the distance of their boundaries is $1/2$. We can therefore find  $\varphi_n\in C^2(\RR^d)$ with
	$$
	\boldsymbol{1}_{C_{R_n-1}(z_n)}\le\varphi_n\le \boldsymbol{1}_{C_{R_n}(y_n)},\quad
	\lvert\nabla \varphi_n \rvert\le 4 .
	$$
	Clearly, 
	$
	(R_n-1)^d\le\lVert \varphi_n\rVert^2\le R_n^d$ and since $\nabla\varphi_n=0$ on $C_{R_n-1}(z_n)$,
	we get
	$\lVert \nabla \varphi_n\rVert^2\le 16\left(R_n^d-(R_n-1)^d\right)$.
	For the normalized $\psi_n:=\lVert \varphi_n\rVert^{-1}\varphi_n$, we obtain
	$$
	(-\Delta\psi_n\mid \psi_n)=\lVert\nabla \psi_n\rVert^2\to 0\;\:\mbox{as}\;\:n\to\infty
	$$
	and
	\begin{align*}
		(V\psi_n\mid\psi_n)&=\int_{C_{R_n}(y_n)}V(x)\psi_n^2(x) \mathrm{d}x\\
		&\le \int_{C_{R_n}(y_n)}V(x)\frac{1}{(R_n-1)^d} \mathrm{d}x\\
		&=\frac{R_n^d}{(R_n-1)^d}\frac{1}{R_n^d}\int_{C_{R_n}(y_n)}V(x)\mathrm{d}x\\
		&<\gamma
	\end{align*}
	for $n$ large enough, so that
	$$
	\inf\sigma(-\Delta +V)\le \inf_{n\in\NN}\left((-\Delta\psi_n\mid \psi_n)+(V\psi_n\mid\psi_n)\right)<\gamma, 
	$$
	which completes the proof. 
	\end{proof}
	\begin{proof}[Proof of Theorem \ref{Theo2.2} -- easy part]
	Since $H+V$ is strictly positive if and only if $\inf\sigma(H+V) >0$,
	it is clear from~(\ref{Dia}) that (i) implies (iii). Furthermore, if (iii) holds and the lower limit is $\varrho >0$,
	then $V$ is obviously $(R,\varrho/2)$-thick for all sufficiently large $R$. 
	This proves the implication (iii) $\Longrightarrow$ (ii). 
	\end{proof}
	It remains to prove the implication (ii) $\Longrightarrow$ (i). 
	Our  proof is based on Theorem~\ref{Theo3.2}, and it also delivers an explicitly computable lower bound; we state the corresponding estimates in the following theorem and its corollary. 
	\begin{thm}\label{Theo4.1}
		If $R > 0$, $C = (0,R)^d$, $V \in L^\infty(C)$, $V \ge 0$, and $H_C = -\Delta_C$ is the Laplacian in $L^2(C)$
		with Neumann boundary conditions on $C$, then
		$$ 
		H_C+V\geq \quad\frac{\pi^2}{\pi^2+2R^2\lVert V\rVert_{L^\infty (C)}} \frac{1}{R^d} \int_CV(x) \mathrm{d}x.
		$$
	\end{thm}
\begin{proof}
If $V=0$ there is nothing to prove. Hence, we assume throughout that $V$ is not identically zero.
Recall that $H_C+V$ is the unique selfadjoint operator associated with the closed quadratic form 
	$$
	W^{1,2}(C)\rightarrow [0,\infty), \quad u\mapsto\int_{C}|\nabla u|^2\mathrm{d}x + \int_{C}V|u|^2 \mathrm{d}x 
	$$
	defined in the Hilbert space $L^2(C)$. From the Poincar\'{e} inequality we know that
	\begin{equation}\label{PI}
		\int_C \lvert \nabla f \rvert^2 \mathrm{d} x\geq \frac{\pi^2}{R^2} \lVert f \rVert^2
	\end{equation}
	provided $f\bot\ker (H_C)=\RR\cdot 1$. The closed space  $\cK:=\ker(H_C)$ is a reducing subspace for $H_C$, 
	and  as $(H_C f | f)$ is equal to the integral in (\ref{PI}), we see that the assumption $(H_C \xi|\xi)\ge c\lVert\xi\rVert^2$ of 
	Theorem~\ref{Theo3.2} is satisfied with $c={\pi^2}/{R^2}$.
	Since $\cK$ consists of constant functions, the assumption $(V\zeta|\zeta) \ge v\lVert\zeta\rVert^2$ of Theorem~\ref{Theo3.2} holds with \
	$$v=\frac{1}{R^d} \int_CV(x) \mathrm{d}x.
		$$
	The assertion now follows from Theorem~\ref{Theo3.2}.
\end{proof}
\begin{cor}\label{Cor4.2}
Let $V \in L^\infty_{\rm loc}(\RR^d)$
and $V \ge 0$. Then
\[
\inf\sigma (-\Delta+V) \ge 
\sup_{R >0}\inf_{y \in R\ZZ^d}\frac{\pi^2}{\pi^2+2R^2 \lVert V\rVert_{L^\infty(C_R(y))}}
\frac{1}{R^d}\int_{C_R(y)}V(x) \mathrm{d}x.
\]
In particular, if $V$ is $(R,\varrho)$-thick, then
\[
\inf\sigma (-\Delta+V) \ge \frac{\pi^2 \varrho}{\pi^2+2R^2 \lVert V\rVert_{L^\infty(\RR^d)}} .
\]
\end{cor}
\begin{proof}
 By Dirichlet-Neumann bracketing, see \cite[Section XIII]{RS78}, we have
	\begin{equation*}
		-\Delta+V\geq  \bigoplus_{k\in R\ZZ^d}(-\Delta+V)_{C_R(k)},
	\end{equation*}
	where $R > 0$ is arbitrary, and where $(-\Delta+V)_{C_R(k)}$\ denotes the restriction of $H+V$\ to $L^2(C_R(k))$\ with Neumann boundary conditions. The assertion now follows from (a shifted version of) Theorem~\ref{Theo4.1} and the fact that
	\[
	\sigma\left( \bigoplus_{k\in R\ZZ^d}(-\Delta+V)_{C_R(k)} \right) = \bigcup_{k \in R \ZZ^d} \sigma \bigl( (-\Delta+V)_{C_R (k)} \bigr) .  \qedhere
	\]
\end{proof}
Corollary~\ref{Cor4.2} gives the implication (ii)$\Rightarrow$(i) of Theorem~\ref{Theo2.2} and the proof of Theorem \ref{Theo2.2} is complete.\qed
\section{Miscellanea}
	\subsection{(\ref{QUP}) does not imply $\boldsymbol{H + V \gg 0}$}
	Here we derive Observation \ref{Theo2.3}, 
	which in particular implies that the title of this subsection is correct. Actually, this is an understatement: 
	while (\ref{QUP}) holds for any nontrivial $V$, we already know that ${H + V \gg 0}$ requires $V$ to be $(R,\rho)$-thick.
\medskip	
For the Fourier transform of $f \in L^2(\mathbb{R}^d)$,
	$$
	\hat{f}(\xi)=(2\pi)^{-d/2}\int f(x)e^{-i\xi x}\mathrm{d}x,$$ it follows that 
	$$(-\Delta f)\hat{~}(\xi)=\lvert \xi \rvert^2\hat{f}(\xi)$$
	and hence $P_E$ is the orthogonal projection onto the subspace of all functions $f\in L^2(\mathbb{R}^d)$ 
	whose Fourier
	transform $\hat{f}(\xi)$ is supported in the ball $\lvert \xi \rvert^2 \le E$. 
	
	\begin{proof}[Proof of Observation \ref{Theo2.3}] In the case at hand,
	$$
	(P_Ef)(x)=(2\pi)^{-\frac{d}{2}}\int_{\lvert \xi \rvert^2\le E}\hat{f}(\xi)e^{i\xi x}d\xi
	$$
	is an analytic function on $\RR^d$  
	for every $f\in L^2(\RR^d)$. Consequently,  
	if $f$ is in the range of $P_E$, then $f$ is analytic on all of $\RR^d$. If in addition $Vf=0$,
	then $f$ must vanish a.e. on $\{ x:V(x)>0\}$, which set has strictly positive measure as soon as $V$ is not $0$ a.e. 
	We can now use the strong form of the identity theorem for multivariable analytic functions
	which says that for a nontrivial analytic function the measure of its zero set is zero; see \cite{Mit2020}.
	This gives that $f=0$. 
	\end{proof}
	
	\subsection{Uncertainty principles for operators with irregular coefficients}\label{sub5.2}
	
	As mentioned earlier, the equivalence in Theorem \ref{Theo2.2} in combination with Theorem \ref{Theo2.1} gives that (\ref{EUP}) holds for all selfadjoint operators satisfying $H\ge \eta(-\Delta)$ with some $\eta>0$ and all $(R,\rho)$-thick potentials. Note that for the comparison of the operators, it is no problem to consider non-densely defined forms, so $H$ could be a selfadjoint operator in some closed subspace of $L^2(\RR^d)$.

	\begin{cor}\label{cor5.1} Let $\eta >0$ and $H$ be a selfadjoint operator in some closed subspace $\mathcal{H}\subset L^2(\RR^d)$, satisfying $H\ge \eta(-\Delta)$ and $V\in L^\infty(\RR^d)$ an $(R,\rho)$-thick potential. Then
	$$
	H+V\ge \min\{\eta, 1\}\frac{\pi^2 \varrho}{\pi^2+2R^2 \lVert V\rVert_{L^\infty(\RR^d)}}=: E_0 .
	$$
	Moreover, for any $E<E_0$, $H$ and $V$ satisfy an \eqref{EUP} at $E$ with $\mu(E)\ge E_0-E$, as well as a \eqref{sUP} at $E$ with $\kappa(E)\ge E-E_0$.
	\end{cor}
\begin{proof}
 By the ellipticity assumption above, it follows that 
 $$
 H\ge \eta(-\Delta),
 $$
 so that 
 $$
 H+V\ge \min\{\eta, 1\}(-\Delta + V),
 $$
 which yields the asserted lower bound by Corollary \ref{Cor4.2} and so in turn the uncertainty estimates at low energies by Theorem~\ref{Theo2.1}. 
\end{proof}

Here is a concrete class of examples that arise in this way.
 Assume that $\Omega\subset \RR^d$ is open, $a(x)\ge \eta$ is a  $d\times d$ matrix for $x\in \Omega$ so that $\Omega\ni x\mapsto a_{i,j}(x)$ is locally integrable for all $i,j= 1, \ldots, d$. Moreover, let $W:\Omega\to[0,\infty]$ be locally integrable. 
	
\begin{rem} (1) Define the selfadjoint operator $H$ in $L^2(\Omega)$ via its closed quadratic form, 
	\begin{align}\label{genop}
	D(\mathbf{h})&=\set{u\in W^{1,2}_0(\Omega)}{(a(\cdot)\nabla u(\cdot)|\nabla u(\cdot)), W|u|^2\in L^1(\Omega)},\\\nonumber
	\mathbf{h}[u]&=\int_\Omega(a\nabla u|\nabla u)\mathrm{d}x+\int_\Omega W|u|^2\mathrm{d}x .
	\end{align}
	Then $H\ge \eta (-\Delta)$ in the sense of (not necessarily densely defined) quadratic forms and the assertion of the previous corollary applies.
	
	(2) In the setting of (1) it is more natural to consider a bounded measurable function $V:\Omega\to [0,\infty)$ instead of a potential defined on all of $\RR^d$.  If $V$ is $(R,\rho)$-thick in the sense that
	\[\frac{1}{R^d}\int_{C_R(y)\cap \Omega} V(x) \mathrm{d}x \ge \varrho\]
	for all $y \in \mathbb{R^d}$ with $C_R(y)\cap \Omega\not=\emptyset$, then we can extend it to a function $\tilde{V}$ on $\RR^d$ by setting $\tilde{V}$ equal to $\varrho$ on $\RR^d\setminus\Omega$. Evidently, $\tilde{V}$ is $(R,\rho)$-thick and $\| \tilde{V}\|_\infty=\| V\|_\infty$ so that the conclusion of the preceding corollary applies for $H$ and $V$.
\end{rem}

We would like to compare our uncertainty estimates to those obtained in \cite{StollmannS-21}. The method in this latter paper is restricted to the case $V=\Eins_D$, where $D$ satisfies a strong relative density condition, strictly more incisive than what we need here. On the other hand, in the more specialized situation, Theorem 1.1 of the latter paper gives a  slightly better energy range and \eqref{sUP} than our result above. At the same time, our proof here is much simpler.

	\subsection{Examples of thick and non-thick potentials}
	We consider $A=-\Delta+V$  on $L^2(\RR^d)$ with $V \ge 0$ in $L^\infty(\RR^d) \setminus\{0\}$.
	In the following we permanently employ Theorem~\ref{Theo2.3}.
	
	\medskip
	Let first $d=1$. To ensure that $A\gg 0$, we must find $R>0$ and $\varrho>0$  such that
	$(1/R)\int_y^{y+R}V(x) \mathrm{d}x \ge \varrho$ for all $y \in \RR$. This is impossible if
	the essential supremum of $V$ over $(y,\infty)$ goes to zero as $y$ goes to infinity,
	which in particular happens if $V$ is piecewise continuous and monotonically decreases
	to zero. Thus, in this case $A$ is not strictly positive. Situations in which $V$ does
	not decrease to zero are more interesting. The function $V(x)=\lvert \sin(x) \rvert$ is $(\pi,2/\pi)$-thick
	and hence induces a strictly positive operator $A$. Clearly, $A \gg 0$ even for every
	periodic potential $V$. Choosing natural numbers $k_1 < k_2 < k_3 < \ldots$ and letting
	$V(x)=1$ on $[k_j,k_j+1] \cup [-k_j,-k_j-1]$ and $V(x)=0$ otherwise, we obtain that 
	$T \gg 0$ if the distances $k_{j+1}-k_j$ remain bounded whereas $A$ is not strictly
	positive if $k_{j+1}-k_j \to \infty$ as $j \to\infty$. To have a last example, let
	$V(x)=1$ for $x$ in $\bigcup_{k\in \ZZ \setminus \{0\}}[k,k+1/\lvert k\rvert^\beta]$ with some
	$\beta >0$ and $V(x)=0$ elsewhere.
	Then $(1/n)\int_1^{n+1} V(x) \mathrm{d}x \to 0$ as $n \to\infty$, implying that $A$ is not strictly positive.
	
	\medskip
	Let now $d \ge 2$. Again $A$ is not strictly positive if the essential supremum of $V$ over 
	$\{y \in \RR^d: \lvert y \rvert \ge R\}$  decays to zero as $R \to\infty$. On the other hand, $A\gg 0$
	if $V$ is periodic, for example, if $V(x)=1$ on the black cubes and $V(x)=0$ on the white
	cubes of an infinite $d$-dimensional chessboard. Finally, take spherical shells 
	$S_k=\{x\in \RR^d: k \le \lvert x \rvert \le k+\delta_k\}$ ($k=1,2, \ldots$) with $\delta_k\le 1$.
	For $\delta_k=1$, the volume $\lvert S_k \rvert$ of $S_k$ is $B_d[(k+1)^d-k^d] > B_ddk^{d-1}$,
	where $B_d$ is the volume of the unit ball in $\RR^d$. Thus, for each $\beta >0$,
	we can choose $\delta_k \in (0,1)$
	so that the shells do not intersect and $\lvert S_k \rvert=B_ddk^{d-1}/k^\beta$. A simple computation reveals
	that $(1/n^d)\sum_{k=1}^{n}\lvert S_k\rvert\to 0$ as $n \to \infty$. Consequently, if $V=1$ on the union
	of the shells $S_k$
	and $V=0$ otherwise, then $A$ is not strictly positive. However, if we take $\delta_k=\mu\in(0,1)$
	for all $k$, then $\lvert S_k\rvert =B_d [(k+\mu)^d-k^\mu] >B_dd\mu k^{d-1}$ and hence $(1/n^d)\sum_{k=1}^{n}\lvert S_k\rvert$
	remains bounded away from zero as $n\to \infty$. Thus, in this case the characteristic function of the union of the
	shells generates a strictly positive operator $A$.

\section{Application to control theory} 

Let us discuss an application of our results to control theory. We assume that $H$ is a selfadjoint operator in $\mathcal{H}\subset L^2(\RR^d)$ as in Subsection 5.2 above, i.e. $H\ge \eta(-\Delta)$ for some $\eta>0$. 
For $T > 0$ we consider the linear control problem
\begin{equation} \label{eq:control1}
	\dot f(t) +Hf(t) = V u(t), \quad t \in (0,T], \quad f(0) = f_0 \in \mathcal{H},
\end{equation} 
where $V$ is the operator of multiplication by a non-negative function $V \in L^\infty (\RR^d)$, and  $u \in L^r ((0,T) ; \mathcal{H})$  with $r \in [1,\infty]$ is the so-called \textbf{control function}. The solution of \eqref{eq:control1} is given by Duhamel's formula
\[
f (t) = \mathrm{e}^{-tH} f_0 + \int_0^t \mathrm{e}^{-(t-s)H} V u(s) \mathrm{d}s, \quad t\in [0,T] .
\]
Given four numbers $\alpha,K\geq 0$, $T > 0$ and $r \in [1,\infty]$, the linear control problem \eqref{eq:control1} is called {\bf cost-uniformly $\boldsymbol{\alpha}$-controllable} in time $T$ with respect to $L^r((0,T);\mathcal{H})$ and cost $K$ if for every $f_0\in \mathcal{H}$ there is a control function $u \in L^r((0,T);\mathcal{H})$ such that 
	\[
	\lVert f(T) \rVert \leq \alpha \lVert f_0 \rVert
	\quad \text{and}\quad 
	\lVert u \rVert_{L^r((0,T);\mathcal{H})} \le K \lVert f_0 \rVert .
	\]

\medskip
Based on earlier works \cite{TrelatWX-20,HuangWW-20,LiuWXY-22}, it has recently been observed in \cite{EgidiGST-24} 
that an uncertainty relation at low energy of the form \eqref{sUP} is sufficient to guarantee $\alpha$-controllability for 
all $\alpha > 0$ and sufficiently large $T > 0$. This is in contrast to earlier results, where even for $\alpha$-controllability an uncertainty relation at all energies was required. Indeed, as a consequence of our Corollary~\ref{cor5.1} in combination with
\cite[Proposition~3.1]{EgidiGST-24} we obtain a so-called weak observability estimate. Note that the first assumption of \cite[Proposition~3.1]{EgidiGST-24} is satisfied by Corollary~\ref{cor5.1}, while its second assumption follows easily from the spectral theorem. This weak observability estimate is in turn equivalent to $\alpha$-controllability, see \cite[Theorem~1]{TrelatWX-20} for the case of Hilbert spaces (i.e.\ $r=2$), and \cite[Theorem~2.8]{EgidiGST-24} for the general framework of Banach spaces. Thus, we obtain the following corollary.
\begin{cor}
	Assume that $V$ is $(R,\varrho)$-thick for some $R,\varrho > 0$ and let $\alpha \in (0,1)$. Then there exist $T_0,K > 0$ such that for all $T \geq T_0$ and all $r \in [1,\infty]$, the linear control problem \eqref{eq:control1} is cost-uniformly $\alpha$-controllable in time $T$ with respect to $L^r((0,T);\mathcal{H})$ and cost $K$.
\end{cor}

Note that in the present general context, we do not obtain 
$\alpha = 0$, a property called null-controllability. This would require an uncertainty relation at all energies. We want to stress the fact that we can treat a general class of operators $H$, merely satisfying $H \geq \eta(-\Delta)$ for some $\eta > 0$; see Section~\ref{sub5.2} for examples. As our examples cover operators with irregular coefficients, it seems (currently) beyond reach to show an uncertainty relation at all energies by standard methods to the best of our knowledge. Note, moreover, that in our previous result, the control operator $V$ is merely supposed to be a $(R,\varrho)$-thick potential, generalizing the case of indicator functions of thick sets treated previously in the literature.

%

%
%

\end{document}